\documentclass{article}
\usepackage[utf8]{inputenc}

\usepackage[backend=bibtex, style=alphabetic, useprefix=true,natbib,maxnames=10,backref=true]{biblatex}
\usepackage[colorlinks]{hyperref}
\usepackage[postdot]{glossaries-extra}

\usepackage{stackengine}
\usepackage{epigraph}
\usepackage{chngcntr}
\renewbibmacro{in:}{}

\usepackage[absolute,overlay]{textpos}

\usepackage{mathpazo, courier}
\usepackage{mathtools,blkarray}
\usepackage{mathtools}
\usepackage{changepage}
\usepackage{titling}
\usepackage{marvosym}
\usepackage{multirow}
\usepackage{lipsum}
\usepackage{listings}
\usepackage{adjustbox}
\usepackage{multicol}
\usepackage{tikz-cd}
\usetikzlibrary{positioning}
\tikzset{main node/.style={circle,fill=blue!20,draw,minimum size=0.1cm,inner sep=0pt}}
\usetikzlibrary{calc}
\usepackage{graphicx}
\usepackage{amssymb}
\usepackage{amsfonts}
\usepackage{qtree}
\usepackage{tipa}
\usepackage{amsthm}
\usepackage{amsmath}
\allowdisplaybreaks
\usepackage{color}
\usepackage{enumerate}
\usepackage{doi}
\usepackage{mathdots}
\usepackage{xcolor}
\usepackage{csquotes}
\usepackage[all,cmtip]{xy}

\usepackage{etoolbox}
\usepackage{longtable}
\usepackage{prerex} 

\newtheorem*{corolario*}{Corollary}
\newtheorem*{proposicion*}{Proposition}
\newtheorem*{teorema*}{Theorem}
\newtheorem{teorema}{Theorem}[section]

\newtheorem{proposicion}[teorema]{Proposition}
\newtheorem{corolario}[teorema]{Corollary}

\theoremstyle{definition}

\newtheorem{recordatorio}[teorema]{Reminder}

\newtheorem{definicion}[teorema]{Definition}

\newtheorem{warning}[teorema]{Warning}

\newtheorem{convencion}[teorema]{Convention}

\newtheorem{ejemplo}[teorema]{Example}

\newtheorem{remark}[teorema]{Remark}

\newtheorem{notacion}[teorema]{Notation}

\DeclareMathOperator{\rec}{rec}
\DeclareMathOperator{\mirror}{mirror}
\DeclareMathOperator{\coeff}{coeff}
\DeclareMathOperator{\id}{id}
\DeclareMathOperator{\exc}{exc}
\DeclareMathOperator{\des}{des}

\title{Palindromicity of multivariate Eulerian polynomials}
\author{Alejandro González Nevado\thanks{Alejandromagyar@gmail.com}}
\date{January 2026}

\begin{document}

\maketitle

\begin{abstract}
We lift to the multivariate Eulerian polynomials the identity implying that univariate Eulerian polynomials are palindromic. As a consequence of this generalization, we obtain nice combinatorial identities that can be directly extracted from this polynomial relation and the bijection between permutations involved in the proof of the identity.
\end{abstract}

\section{Introduction}

The Eulerian polynomials are well known in combinatorics because they admit many definitions coming from counting different features in permutations. Here we will use the definition using descents because it generalizes easily to the multivariate case.

\begin{definicion}\cite[Section 2.2]{haglund2012stable}[Univariate Eulerian polynomials]
Let $\mathfrak{S}_{n+1}$ be the symmetric group considered as permutations on the ordered set $[n+1]$ and written in its one-line form. Then we define the univariate polynomial $$a_{n}(x)=\sum_{\sigma\in \mathfrak{S}_{n+1}}x^{\des(\sigma)}\in\mathbb{N}[x]$$ as the generating polynomial for the descent statistic over the symmetric group $\mathfrak{S}_{n+1}$.
\end{definicion}

This polynomial admits a very natural multivariate lifting. We will use the one-line notation for permutations because it is the most comfortable in this setting. We first need to count descents and ascents in a finer way.

\begin{definicion}[Tops sets]\cite[Subsection 2.1]{haglund2012stable}
and \cite[Definition 3.1]{visontai2013stable}\label{dtat}
For $\sigma\in\mathfrak{S}_{n+1}$, we define the \textit{descent top set} $$\mathcal{DT}=\{\sigma_{i}\in[n+1]\mid i\in[n],\sigma_{i}>\sigma_{i+1}\}\subseteq[2,n+1].$$ Similarly, we define the \textit{ascent top set} $$\mathcal{AT}=\{\sigma_{i+1}\in[n+1]\mid i\in[n],\sigma_{i}<\sigma_{i+1}\}\subseteq[2,n+1].$$ 
\end{definicion}

Now this allows us to lift the univariate polynomials to a multivariate setting. We do this tagging the variables through the corresponding indices at which descents and ascents, respectively, happen.

\begin{definicion}\cite[Theorem 3.2]{haglund2012stable}
and \cite[Theorem 3.3]{visontai2013stable}\label{multieulerian}
Let $$A_{n}(\mathbf{x},\mathbf{y}):=\sum_{\sigma\in\mathfrak{S}_{n+1}}\prod_{i\in\mathcal{DT}(\sigma)}x_{i}\prod_{j\in\mathcal{AT}(\sigma)}y_{j}$$ be the \textit{descents-ascents $n$-th Eulerian polynomial}.
\end{definicion}

Here we will mainly concentrate in the descents polynomial $A_{n}(\mathbf{x},\mathbf{1}).$ We will denote it just $A_{n}(\mathbf{x}).$ These polynomials are interesting for different areas of mathematics because they have very nice properties and appear in different contexts naturally. Here we want to prove a combinatorial property that generalizes the palindromicity of the univariate polynomials to the multivariate case.

\section{Preliminaries}

Eulerian polynomials are well-know to verify a palindromic identity. The reciprocal of a univariate polynomial is classically known.

\begin{definicion}[Reciprocal polynomial]\cite[Section 2.1]{joyner2018self}
Let $p\in\mathbb{R}[x]$ be a univariate polynomial. We denote $$\rec(p(x)):=\rec(p)(x):=x^{\deg(p)}p(\frac{1}{x})$$ its \textit{reciprocal}.
\end{definicion}

Thus when one studies multivariate extensions of Eulerian polynomials, it is natural to ask if we can construct these identities also for the reciprocal $\rec(A_{n})$ of $A_{n}$. It turns out that the direct approach works in this case. We remind an important concept of symmetry for univariate polynomials.

\begin{definicion}[Palindromicity as symmetry]\cite[Section 2.1]{joyner2018self}
Let $p\in\mathbb{R}[x]$ be a univariate polynomial. We say that $p$ is \textit{palindromic} if $p=\rec(p).$
\end{definicion}

In particular, the univariate Eulerian polynomials verify this symmetry.

\begin{teorema}
Let $a_{n}$ be the $n$-th univariate Eulerian polynomial. Then $a_{n}=\rec(a_{n}).$ This means that univariate Eulerian polynomials are palindromic.
\end{teorema}

\begin{proof}
As Eulerian polynomials have a nice combinatorial description in terms of descents of permutations, a simple count (that can elegantly be performed through a bijection) over the descent statistic of permutations shows this.
\end{proof}

The mentioned proof strategy resorting to combinatorics and bijections between permutations will be important for us. In fact, we will follow the same strategy for the multivariate case.

\section{Reciprocal notion for multivariate polynomials}

There are many options to extend the reciprocal operation in a natural way to multivariate polynomials. Here we choose the following definition.

\begin{definicion}[Multivariate reciprocal]
\label{reciprocal}
Let $A(\mathbf{x})\in\mathbb{R}[\mathbf{x}]$ be a multivariate polynomial strictly in the variables $\mathbf{x}$. We define its \textit{reciprocal} $$\rec(A(\mathbf{x}))=\rec(A)(\mathbf{x}):=m_{A}A(\frac{1}{x_{1}},\dots,\frac{1}{x_{n}}),$$ where $m_{A}$ is the \textit{only} minimum degree monic monomial in the variables $\mathbf{x}$ that produces a polynomial after performing the multiplication.
\end{definicion}

For multiaffine polynomials, the monomial $m_{A}$ is clear once we fix the variables. A multiaffine polynomial presenting such monomial reaches a kind of maximality. First, we need to establish a natural convention.

\begin{convencion}
From now on we say that a polynomial is a polynomial \textit{strictly} in the variables $\mathbf{x}$ if it is not possible to write it using a subset of these variables. We then say that each variable in $\mathbf{x}$ is \textit{strictly appearing} in the polynomial.
\end{convencion}

This convention, in particular, will be important for us due to the natural presence of a ghost variable in multivariate Eulerian polynomials. This ghost variable appears because $1$ can never be a descent top and therefore the variable $x_{1}$ does not appear so that $A_{n}(x_{1},\dots,x_{n+1})$ is actually a degree $n$ polynomial in the $n$ variables $(x_{2},\dots,x_{n+1})$. Now, having in mind this convention, we can describe our form of maximality for multiaffine polynomials.

\begin{definicion}[Monomialmaximality]\label{monomax}
We call a multiaffine polynomial in the variables $\mathbf{x}$ \textit{monomialmaximal} if the coefficient of the monomial of maximum degree that can appear in such a multiaffine polynomial not adding to it new strictly appearing variables is not $0$.
\end{definicion}

Clearly, this means that if such polynomial is strictly actually just in the variables $(x_{i_{1}},\dots,x_{i_{m}})$ the monomial $\prod_{j=1}^{m}x_{i_{j}}$ has not $0$ as its coefficient. As we wanted, our polynomials of interest are monomialmaximal.

\begin{proposicion}[Eulerian monomialmaximality]
(Multivariate) Eulerian polynomials $A_{n}(\mathbf{x})$ are monomialmaximal.
\end{proposicion}

\begin{proof}
In the definition of $A_{n}$ consider the term of the sum generated by the permutation (in one-line notation) $\sigma:=(n+1\dots1)\in\mathfrak{S}_{n+1}$. This term generates the maximal monomial possible in this multiaffine polynomial as $\sigma_{i}>\sigma_{i+1}$ for all $i\in[n]$ and, therefore, $A_{n}$ is monomialmaximal.
\end{proof}

Monomialmaximal multiaffine polynomials are nice with respect to taking reciprocals. Thus, the following result now makes it easy to continue our path.

\begin{proposicion}[Multiaffine monomialmaximal reciprocals] Let $A(\mathbf{x})\in\mathbb{R}[\mathbf{x}]$ be a multiaffine monomialmaximal polynomial with $A(\mathbf{0})\neq0$. Then its reciprocal polynomial $\rec(A)(\mathbf{x})$ is also multiaffine and monomialmaximal with $\rec(A)(\mathbf{0})\neq0$.
\end{proposicion}

\begin{proof}
We begin writing $$A(\mathbf{x})=\sum_{\alpha\subseteq\{1,\dots,n\}}a_{\alpha}\mathbf{x}^{\alpha}.$$ We know that $a_{\{1,\dots,n\}}$ is not zero as the polynomial is monomialmaximal. Therefore, $m_{A}$ in the definition of the reciprocal is forced to be the degree-maximal possible monomial in a multiaffine polynomial $\mathbf{x}^{1,\dots,n}$. In sum, all this means that the reciprocal $$\rec(A)(\mathbf{x}):=\mathbf{x}^{\{1,\dots,n\}}A(\frac{1}{x_{1}},\dots,\frac{1}{x_{n}})=\sum_{\alpha\subseteq\{1,\dots,n\}}a_{\alpha}\mathbf{x}^{\{1,\dots,n\}\smallsetminus\alpha}.$$ As $a_{\{1,\dots,n\}}\neq0$ this monomial has nonzero independent term. Similarly, as $A(\mathbf{0})\neq0$, $a_{\emptyset}\neq0$ so $\rec(A)$ is monomialmaximal. Multiaffinitty follows directly from how we wrote the polynomial in set exponent notation.
\end{proof}

In fact, for Eulerian polynomials something else can be said because this transformation actually gives back the same Eulerian polynomial but with the variables permuted. We will see this formally in the next section. Before, we need a reminder.

\begin{recordatorio}[Ghost variables]
Contrary to what we said above, the definition of Eulerian polynomials involve a ghost variable $x_{1}$ which is actually not a variable of the polynomial. However, our arguments will follow in an easier way if we deal with this variable.
\end{recordatorio}

We make this comment and add another warning before we work out the next proof.

\begin{warning}[Definition of cofactor]
Additionally, observe that we took care of this phenomenon in Definition \ref{reciprocal} when we defined $m_{A}$ literally as \textit{the minimum degree monomial in the variables $\mathbf{x}$ that produces a polynomial after performing the multiplication} $m_{A}A(\frac{1}{x_{1}},\dots,\frac{1}{x_{n}})$. This ensures that, even when $\mathbf{x}$ contains ghost variables not appearing at all in $A$, the correcting multiplicative monomial $m_{A}$ will not be affected by that arbitrariness and it is therefore well-defined. We also took care of this when defining monomialmaximality in Definition \ref{monomax}.
\end{warning}

Now we are in position to analyze how these polynomials are built. For this, we look at possible patterns in order to find internal symmetries on them resembling (univariate) palindromicity. This is the topic of the next section.

\section{Fixing subsets of permutations}

We look at permutations and at the sets they fix in order to find the symmetries we search for. For this, we restrict our view to permutations verifying certain conditions at a time.

\begin{notacion}[Mirrors and conditioned permutations]
Let $\mathbf{x}=(x_{1},\dots,x_{n})$, then we denote $\mirror(\mathbf{x}):=(x_{n},\dots,x_{1}).$ For a set of permutations $A\subseteq\mathfrak{S}_{n}$ and two elements $a,b\in[n]$ we denote $A^{a\to b}$ the subset of permutations $\sigma\in A$ sending $a$ to $b$ $$A^{a\to b}:=\{\sigma\in A\mid\sigma(a)=b\}.$$
\end{notacion}

Using the mirror map one can generalize the notion of palindromicity to multivariate polynomials in the sense we pursue here.

\begin{definicion}[Mirrorpalindromicity]
We say that a polynomial $p(\mathbf{x})\in\mathbb{R}[\mathbf{x}]$ is \textit{mirrorpalindromic} if its reciprocal polynomial verifies $$\rec(p)(\mathbf{x})=p(\mirror(\mathbf{x})).$$
\end{definicion}

We will need some additional notation for the proof of the next theorem.

\begin{definicion}[Cuts, splitters and set actions in one-line notation]
Given a permutation $\sigma=(\sigma_{1}\cdots\sigma_{n})\in\mathfrak{S}_{n}$ in the one-line notation, we understand that we can introduce \textit{markers} splitting (or \textit{splitters}) the elements that will be generally commas so that we can write $\sigma=(\sigma_{1}\cdots\sigma_{k},\sigma_{k+1}\cdots\sigma_{n}).$ However, if we have more information about the relations between $\sigma_{k}$ and $\sigma_{k+1}$ we can express this directly in this extended one-line notation using an inequality symbol as separation so that we can write $\sigma=(\sigma_{1}\cdots\sigma_{k}>\sigma_{k+1}\cdots\sigma_{n})$ if we know that  $\sigma_{k}>\sigma_{k+1}$ (and also the other way around if the inequality is reversed). Thus, we signal a descent that we know that happens for sure for that permutation between position $k$ and position $k+1$. In particular, we signal thus that $\sigma_{k}$ is a descent top. Sometimes we will use such information to \textit{cut} the one-line notation through some element. Thus we say that such element establishes a \textit{cut through it}. Thus if we have $\sigma=(\sigma_{1}\cdots\sigma_{k-1} \sigma_{k} \sigma_{k+1}\cdots\sigma_{n})$ and we \textit{cut through} $\sigma_{k}$ this means that we will now consider the formal one-line notations $(\sigma_{1}\cdots\sigma_{k-1})$ and $(\sigma_{k+1}\cdots\sigma_{n}),$ which clearly do not in general need to be permutations by themselves. We will however concatenate these cuts and the cutting elements to form new one-line notations actually representing permutations. Finally, when we have a set $A\subseteq\mathfrak{S}_{n}$ and a formal one-line notation $(\tau_{1}\cdots\tau_{m})$ of elements in $[n+1,n+m]$ we can concatenate them using a splitter to form a \textit{lifting} of the set $A$ injected into the set of permutations $\mathfrak{S}_{n+m}$ defining $(A(\tau_{1}\cdots\tau_{m})):=\{(\sigma_{1}\cdots\sigma_{n}\tau_{1}\cdots\tau_{m})\in \mathfrak{S}_{n+m}\mid (\sigma_{1}\cdots\sigma_{n})\in A\}.$
\end{definicion}

We will use a particular collection of splitters in our proof. We particularize now to our case of interest in order to ensure higher clarity.

\begin{remark}[Splitters and non-splitters]
Notice that $=$ cannot be used as an splitter and if it appears it just serves the purpose of writing the same element on a different way because all the elements in the one-line notation have to be different. Notice also, in particular, that, when we lift a set of permutations in the definition above, we can always use the splitter $<$ and write $(A<(\tau_{1}\cdots\tau_{m})):=\{(\sigma_{1}\cdots\sigma_{n}<\tau_{1}\cdots\tau_{m})\in \mathfrak{S}_{n+m}\mid (\sigma_{1}\cdots\sigma_{n})\in A\}$ because $\sigma_{i}<\tau_{j}$ for any $i\in[1,n]$ and $j\in[1,m].$
\end{remark}

\section{Counting for clarity}

We will introduce notation for sets of permutations in terms of their sets of descents.

\begin{definicion}[Exact descent]
\label{rns}
Let $S\subseteq[n+1]$. We introduce the set of permutations that \textit{descend exactly} at $S$, i.e., $$R(n,S):=\{\sigma\in\mathfrak{S}_{n+1}\mid S=\mathcal{DT}(\sigma)\}.$$
\end{definicion}

Reference to $n$ can be dropped when $n$ is big enough or understood by the context and therefore we might simply write $R(S)$ instead sometimes

\begin{remark}[Limited cardinality and shorthands]\label{rshortening}
When $n$ is fixed we simply denote these numbers $R(S)$ for any subset $S\subseteq[n+1].$
\end{remark}

\begin{notacion}\label{defe}
Let $X\subseteq[n+1]$ we denote $$X(\mathfrak{S}_{n+1}):=\{\sigma\in\mathfrak{S}_{n+1}\mid\mathcal{DT}(\sigma)\subseteq X\}.$$
\end{notacion}

Computing the cardinal of these sets require the use of two maps on sets and tuples.

\begin{notacion}[Operators]\cite[Appendix A]{davis2018pinnacle}
For an ordered set $X=\{x_{1}<\dots<x_{k}\}$, we define the operator $$\alpha(X)=(x_{1}-1,x_{2}-x_{1},x_{3}-x_{2},\dots,x_{k}-x_{k-1}).$$ And for a tuple $\beta=(\beta_{1},\dots,\beta_{k})$ we denote by $$\beta\hat{!}=(k+1)^{\beta_{1}}k^{\beta_{2}}\cdots3^{\beta_{k-1}}2^{\beta_{k}}.$$
\end{notacion}

For completeness, we convene $\alpha(\emptyset)=()$ and $()\hat{!}=1$. With this, we are ready to find the cardinal of the sets of permutations considered in Notation \ref{defe}.

\begin{proposicion}\cite[Theorem A.1]{davis2018pinnacle}
Let $X\subseteq[n]$. Then $$|X(\mathfrak{S}_{n})|=\alpha(X)\hat{!}$$
\end{proposicion}

We then obtain the following expression for the cardinals $|R(n,X)|.$

\begin{corolario}[Cardinal of sets with exact descent top set]\cite[Appendix A]{davis2018pinnacle}
\label{coroR}
We can express the number $|R(n-1,X)|$ in terms of deletions in the initial set as
\begin{gather}\label{coroR2}
    |R(n-1,X)|=\sum_{J\subseteq X}(-1)^{|X\smallsetminus J|}\alpha(J)\hat{!}.
\end{gather}
\end{corolario}

\section{Multivariate Eulerian mirrorpalindromicity}

Here one has to handle ghost variables adequately.

\begin{notacion}
Due to $x_{1}$ being a ghost variable, denote, during the next theorem and its associated proof, $\mathbf{x}=(x_{2},\dots,x_{n+1})$ for brevity.
\end{notacion}

Analyzing the descent top sets we can now provide a proof of our main theorem.

\begin{teorema}[Eulerian mirrorpalindromicity]\label{especialita}
The $n$-th multivariate Eulerian polynomial $A_{n}(\mathbf{x})$ is mirrorpalindromic.
\end{teorema}

\begin{proof}
Following the definitions of these polynomials, we have that \begin{gather*}\rec(A_{n})(\mathbf{x}):=\mathbf{x}^{\mathbf{1}}\sum_{\sigma\in\mathfrak{S}_{n+1}}\prod_{i\in\mathcal{DT}(\sigma)}\frac{1}{x_{i}}=\sum_{\sigma\in\mathfrak{S}_{n+1}}\prod_{i\in[2,n+1]\smallsetminus\mathcal{DT}(\sigma)}x_{i}\end{gather*} and, introducing the permutation $$\tau\colon[n+2]\to[n+2],i\mapsto\tau(i):=n+3-i$$ and calling $\mathbf{y}=\mirror(\mathbf{x})$ so that $y_{i}=x_{\tau(i)},$ that $A_{n}(\mirror(\mathbf{x}))=A_{n}(\mathbf{y})=$\begin{gather*}
   \sum_{\sigma\in\mathfrak{S}_{n+1}}\prod_{i\in\mathcal{DT}(\sigma)}y_{i}=\sum_{\sigma\in\mathfrak{S}_{n+1}}\prod_{i\in\mathcal{DT}(\sigma)}x_{\tau(i)}.\end{gather*} Now we analyze these polynomials coefficient-wise having in mind that, by the previous propositions, they are multiaffine. Thus, fixing a set $K\subseteq[2,n+1]$, we have that \begin{gather*}\coeff(\mathbf{x}^{K},\rec(A_{n})(\mathbf{x}))=|\{\sigma\in\mathfrak{S}_{n+1}\mid K= [2,n+1]\smallsetminus\mathcal{DT}(\sigma)\}|\\=|\{\sigma\in\mathfrak{S}_{n+1}\mid \mathcal{DT}(\sigma)=[2,n+1]\smallsetminus K\}|:=|R(L)|\end{gather*} with $L:=[2,n+1]\smallsetminus K$ while \begin{gather*}\coeff(\mathbf{x}^{K},A_{n}(\mirror(\mathbf{x})))=|\{\sigma\in\mathfrak{S}_{n+1}\mid K=\tau(\mathcal{DT}(\sigma))\}|=\\|\{\sigma\in\mathfrak{S}_{n+1}\mid \tau(K)=\mathcal{DT}(\sigma)\}|=|R(\tau(K))|,\end{gather*} as $\tau\circ\tau=\id$. Observe that it is in general not evident that these cardinalities are equal as the sets at which the operator $R$ is evaluated in each case are different (even their cardinalities are so: respectively, $n-|K|$ and $|K|$). In order to see that these cardinalities are in fact equal we will establish a bijection between the sets $R(\tau(K))$ and $R(L)$. The first observation that we have to make in our way towards the construction of the mentioned bijection is a possibly oversight fact of Corollary \ref{coroR}. We need to consider this fact because, although we finish if we establish a bijection between the set of permutations with descent top $L$ and the set of permutations with descent top $\tau(K)$, the construction of such bijection between permutations belonging to $\mathfrak{S}_{n+1}$ will require us to first lift them to $\mathfrak{S}_{n+2}$, as a consequence of the fact that $\tau([1,n+1])=[2,n+2]$. Understanding why we are able (and it is therefore useful) to do that will require a close look at the formula in Corollary \ref{coroR}. Observe that the cardinal $|R(n,X)|$ of $R(n,X)$ does not in fact depend on $n$ and we could then write instead that for, all $k\geq\max(X)-1$, we have $$|R(k,X)|=\sum_{J\subseteq X}(-1)^{|X\smallsetminus J|}\alpha(J)\hat{!}.$$ This means that the number of permutations in $\mathfrak{S}_{n+1}$ with descent top $L$ equals the number of permutations with descent top $L$ in $\mathfrak{S}_{n+2}$. This also means that we can biject one set into the other. In particular, we want this to happen through the restriction of the obvious injection $$\mathfrak{S}_{n+1}\to\mathfrak{S}_{n+2},\sigma\mapsto (\sigma\ n+2)$$ to the set $R(n,L)\subseteq\mathfrak{S}_{n+1}$, as the restriction of this injection verifies $$R(n,L)\to R(n+1,L)\subseteq\mathfrak{S}_{n+2}^{n+2\to n+2}$$ bijectively because it is obvious that the injection does not modify the descent top sets and we have just seen above in this proof, through the careful observation of Equation \ref{coroR2} of Corollary \ref{coroR}, that $|R(n,L)|=|R(n+1,L)|$. Therefore this bijection tells us how we want to lift our permutations in $\mathfrak{S}_{n+1}$ to $\mathfrak{S}_{n+2}$ preserving the descent top set, which will be necessary to make sense of the bijection involving the use of $\tau$ that we will construct because $\tau([1,n+1])=[2,n+2].$ In particular, observe that preserving the descent top \textit{requires} the new biggest element $n+2$ of the image permutation to be set at the end in the one-line notation (i.e., to be fixed by the image permutation) because, otherwise, $n+2$ in any other position would be a descent top that would have to be added to the set of descent tops of the transformation. Symmetrically (as $\tau(n+2)=1$), we need to analyze another very important element of these permutations: the position of the never-descent-top element $1$. This element establishes a fundamental cut in the permutation and, moreover, verifies $\tau(1)=n+2$, which gives us a helpful clue about how our desired bijection should act about it. In order to see this, we establish that we will send a lifted permutation $$\mathfrak{S}_{n+2}^{n+2\to n+2}\ni(\sigma < n+2)=(\sigma_{1}\cdots\sigma_{k-1}>1<\sigma_{k+1}\cdots\sigma_{n+1} < n+2)$$ with $1$ at position $k$ and $n+2$ at the end to the new permutation $$(\sigma_{k+1}\cdots\sigma_{n+1}<n+2>\sigma_{1}\cdots\sigma_{k-1}>1)\in\mathfrak{S}_{n+2}^{n+2\to 1},$$ where we remark that the sequences $(\sigma_{1}\cdots\sigma_{k-1})$ and $(\sigma_{k+1}\cdots\sigma_{n+1})$ could be empty if $1$ is at the beginning or at the end, respectively, of $\sigma\in\mathfrak{S}_{n+1}.$ Observe that this is a bijection $\mathfrak{S}_{n+2}^{n+2\to n+2}\to\mathfrak{S}_{n+2}^{n+2\to 1}$ whose domain is precisely the subset of $\mathfrak{S}_{n+2}$ that coincides with our respecting-descent-top-sets immersion of $\mathfrak{S}_{n+1}$ in $\mathfrak{S}_{n+2}$ as $(\mathfrak{S}_{n+1}<n+2)$. Observe, furthermore, that this last bijection sends the elements with descent top $L\subseteq[2,n+1]$ of $\mathfrak{S}_{n+2}^{n+2\to n+2}$ (and thus of $\mathfrak{S}_{n+2}$ as this forces $n+2$ to be fixed) to elements of $\mathfrak{S}_{n+2}^{n+2\to 1}$ whose descent top set is $L\cup\{n+2\},$ adding thus only $n+2$ as a \textit{forced} descent top (as it has to be so as soon as it is not fixed, i.e., at the end in the one-line notation). Thus, so far, we have defined a bijection $$R(n,L)\to R(n+1,L\cup\{n+2\})\cap\mathfrak{S}_{n+2}^{n+2\to 1}.$$ Beware that $R(n+1,L\cup\{n+2\})\cap\mathfrak{S}_{n+2}^{n+2\to 1}\neq R(n+1,L\cup\{n+2\})$ as a cut through any ascent produces elements of $R(n+1,L\cup\{n+2\})$ not in the intersection, e.g., $$(\mu_{1}\cdots\mu_{m-1}<\mu_{m}\cdots\mu_{n+1}> 1)\in\mathfrak{S}_{n+2}^{n+2\to 1}\in R(n+1,L\cup\{n+2\})\cap\mathfrak{S}_{n+2}^{n+2\to 1}$$ produces, after a cut by the highlighted ascent, the permutation $$R(n+1,L\cup\{n+2\})\ni(\mu_{m}\cdots\mu_{n+1}>1<\mu_{1}\cdots\mu_{m-1})\notin\mathfrak{S}_{n+2}^{n+2\to 1}.$$ Once this is clear, now we apply $\tau$ to all the elements in the one-line notation of the image under the previous bijection to obtain the new permutation $$(\tau(\sigma_{k+1})\cdots\tau(\sigma_{n+1})>\tau(n+2)=1<\tau(\sigma_{1})\cdots\tau(\sigma_{k-1})>\tau(1)=n+2).$$ This defines a bijection $$R(n+1,L\cup\{n+2\})\cap\mathfrak{S}_{n+2}^{n+2\to 1}\to R(n+1,\tau(K))\cap\mathfrak{S}_{n+2}^{n+2\to n+2}=R(n+1,\tau(K))$$ because the descents tops of an element in the domain appears as maxima of descents of the form $n+2>\sigma_{1}$ or $\sigma_{l}>\sigma_{l+1}$ and these are transformed by $\tau$ into the ascents (as $\tau$ is order reversing) $1<\tau(\sigma_{1})$ and $\tau(\sigma_{l})<\tau(\sigma_{l+1})$ while the rest of the elements appear either as minima in ascents of the form $\sigma_{l}<\sigma_{l+1}$ or at the end (but at the end we have always $1$) and these are transformed by $\tau$ into descents $\tau(\sigma_{l})>\tau(\sigma_{l+1})$ or in $n+2$. Thus the descent top of the image is $$[2,n+1]\smallsetminus L=[2,n+1]\smallsetminus ([2,n+1]\smallsetminus K)=K.$$ Thus the descent top set of this new permutation is $K=[2,n+1]\smallsetminus L$ because $\tau$ is order reversing and we took the care of having a domain whose permutations have $1$ at then, which transforms into $n+2$ at the end, an element we do not desired to be in our descent top. Observe that, as $\tau\circ\tau=\id$, similar arguments allow for the construction of an inverse for this map, which shows that it is in fact a bijection whose image is the whole lifting $R(n+1,\tau(K))$ of $R(n,\tau(K))$. We now just have to undo the lifting. Finally, we cut out the fix element $n+2$ so we can see this last permutation as $$(\tau(\sigma_{k+1})\cdots\tau(\sigma_{n+1})>\tau(n+2)=1<\tau(\sigma_{1})\cdots\tau(\sigma_{k-1}))\in\mathfrak{S}_{n+1},$$ this establish the whole path of the desired bijection and therefore finishes our proof.\end{proof}

The proof of this last theorem has interesting consequences for several combinatorial sums. We collect the implied identities in the next section. We also have to mention that, in general, finding the kind of bijection between permutations that we found in the proof above is an interesting topic by itself that has been previously studied, e.g., in \cite{bigeni2016new,bloom2020revisiting,chen2024bijection} for bijections conserving other properties or accomplishing or involving different relational counting arguments.

\section{Mirrorpalindromicity from the permutation point of view}

Now we can translate many insights obtained in the last proof of the previous section into the purely combinatorial setting. This gives us results about permutations in the form of the three following corollaries.

\begin{corolario}[Combinatorial sums]
For every $n\in\mathbb{N}$ and every $K\subseteq[2,n+1]$ we have that the next quantities with $\tau\colon[n+2]\to[n+2],i\mapsto \tau(i):=n+3-i$ are equal: $|R([2,n+1]\smallsetminus K)|=$ \begin{gather*}
\sum_{J\subseteq[2,n+1]\smallsetminus K}(-1)^{|[2,n+1]\smallsetminus (K\cup J)|}\alpha(J)\hat{!}=\sum_{J\subseteq[2,n+1]\smallsetminus K}(-1)^{n- |K\cup J|}\alpha(J)\hat{!}=\\\sum_{J\subseteq[2,n+1]\smallsetminus K}(-1)^{n- |\tau(K)\cup \tau(J)|}\alpha(J)\hat{!}=\sum_{\tau(J)\subseteq[2,n+1]\smallsetminus \tau(K)}(-1)^{n- |\tau(K)\cup\tau(J)|}\alpha(J)\hat{!}=\\\sum_{S\subseteq[2,n+1]\smallsetminus \tau(K)}(-1)^{n-|\tau(K)\cup S|}\alpha(\tau(S))\hat{!}=\sum_{S\subseteq\tau(K)}(-1)^{|\tau(K)\smallsetminus S|}\alpha(S)\hat{!}=\\\sum_{\tau(J)\subseteq \tau(K)}(-1)^{|\tau(K)\smallsetminus \tau(J)|}\alpha(\tau(J))\hat{!}=\sum_{J\subseteq K}(-1)^{|\tau(K)\smallsetminus \tau(J)|}\alpha(\tau(J))\hat{!}=\\\sum_{J\subseteq K}(-1)^{|K\smallsetminus J|}\alpha(\tau(J))\hat{!}=\sum_{\tau(J)\subseteq K}(-1)^{|K\smallsetminus \tau(J)|}\alpha(J)\hat{!}=|R(\tau(K))|.
\end{gather*} Additionally, as we have that the identities above are true for all $K\subseteq[2,n+1]$ and $\tau$ induces a permutation on the set of parts $2^{[2,n+1]}$ of $[2,n+1]$, we have, for all $K\subseteq[2,n+1]$, equivalently the identity $$\sum_{J\subseteq[2,n+1]\smallsetminus K}(-1)^{n-|K\cup J|}\alpha(\tau(J))\hat{!}=\sum_{J\subseteq K}(-1)^{|K\smallsetminus J|}\alpha(J)\hat{!}.$$
\end{corolario}

\begin{proof}
This is a direct consequence of the counting formulas collected in Corollary \ref{coroR} and the identity proved in Theorem \ref{especialita}.
\end{proof}

As a consequence of the observations we made, we can say more about these quantities when we increase the $n$ step by step. For this, we need to generalize $\tau$, which actually depends on $n$, and write $\tau_{s}(i):=s+3-i.$

\begin{corolario}[Sequential combinatorial sums]
Let $n\in\mathbb{N}$ and $K\subseteq[2,n+1]$ and for every set $W\subseteq[2,n+1]$ denote $W^{\max}=W\cup\{n+2\}$ and $W^{\min}=W\cup\{1\}$. Then $|R([2,n+2]\smallsetminus K^{\max})|=|R([2,n+1]\smallsetminus K)|$ and $$|R([2,n+2]\smallsetminus K)|=\sum_{J\subseteq[2,n+1]\smallsetminus K}(-1)^{n-|K|-|J|}(\alpha(J^{\max})\hat{!}-\alpha(J)\hat{!})=$$ $$\sum_{J\subseteq[2,n+1]\smallsetminus \tau_{n}(K)}(-1)^{n-|K|-|J|}(\alpha(\tau_{n}(J^{\min}))\hat{!}-\alpha(\tau_{n}(J))\hat{!})=$$ $$\sum_{J\subseteq K}(-1)^{|K|-|J|}\alpha(\tau_{n+1}(J))\hat{!}=$$ $$\sum_{J\subseteq K}(-1)^{|K|-|J|}(|J|+1)\alpha(\tau_{n}(J))\hat{!}=|R(\tau_{n+1}(K))|.$$
\end{corolario}

\begin{proof}
The first equality is evident because $[2,n+2]\smallsetminus K^{\max}=[2,n+1]\smallsetminus K$. For the second we will use Theorem \ref{especialita} above that ensures $|R(\tau_{n+1}(K))|=|R([2,n+2]\smallsetminus K)|$ together with the equalities coming through the use of Corollary \ref{coroR} that tells us $|R([2,n+2]\smallsetminus K)|=$\begin{gather*}
\sum_{J\subseteq[2,n+2]\smallsetminus K}(-1)^{n+1-|K\cup J|}\alpha(J)\hat{!}=\\\sum_{J\subseteq[2,n+1]\smallsetminus K}(-1)^{n+1-|K\cup J|}\alpha(J)\hat{!}+\sum_{J\subseteq[2,n+1]\smallsetminus K}(-1)^{n+1-|K\cup J'|}\alpha(J^{\max})\hat{!}=\\\sum_{J\subseteq[2,n+1]\smallsetminus K}(-1)^{n-|K\cup J|}\alpha(J^{\max})\hat{!}-\sum_{J\subseteq[2,n+1]\smallsetminus K}(-1)^{n-|K\cup J|}\alpha(J)\hat{!}=\\\sum_{J\subseteq[2,n+1]\smallsetminus K}(-1)^{n-|K\cup J|}(\alpha(J^{\max})\hat{!}-\alpha(J)\hat{!})=\\\sum_{J\subseteq[2,n+1]\smallsetminus\tau_{n}(K)}(-1)^{n-|\tau_{n}(K)\cup J|}(\alpha(\tau_{n}(J^{\min}))\hat{!}-\alpha(\tau_{n}(J))\hat{!})
\end{gather*} and $|R(\tau_{n+1}(K))|=$ \begin{gather*}\sum_{J\subseteq K}(-1)^{|K\smallsetminus J|}\alpha(\tau_{n+1}(J))\hat{!}=\sum_{J\subseteq K}(-1)^{|K\smallsetminus J|}(|J|+1)\alpha(\tau_{n}(J))\hat{!}\end{gather*} noting that, writing $J=\{j_{1}<\dots<j_{k}\}$, we have $$\tau_{n}(J)=\{\tau_{n}(j_{k}),\dots,\tau_{n}(j_{1})\}=\{n+3-j_{k}<\dots<n+3-j_{1}\} \mbox{\ and}$$ $$\tau_{n+1}(J)=\{n+4-j_{k}<\dots<n+4-j_{1}\} \mbox{\ so}$$ $$\alpha(\tau_{n}(J))=(n+2-j_{k},j_{k}-j_{k-1},\dots,j_{2}-j_{1})\mbox{\ and}$$ $$\alpha(\tau_{n+1}(J))=(n+3-j_{k},j_{k}-j_{k-1},\dots,j_{2}-j_{1})$$ because we remember that $\alpha(J):=(j_{1}-1,j_{2}-j_{1},\dots,j_{k}-j_{k-1})$ so we have that $\alpha(\tau_{n+1}(J))\hat{!}=(k+1)\alpha(\tau_{n}(J))=(|J|+1)\alpha(\tau_{n}(J)).$
\end{proof}

Reordering and rewriting the expressions obtained above we get another interesting identity.

\begin{corolario}[Reordering sums]
For $n\in\mathbb{N}$ and $K\subseteq[2,n+1]$ we have that $$\sum_{J\subseteq[2,n+1]\smallsetminus K}(-1)^{n-|K|+|J|}\alpha(J^{\max})\hat{!}=\sum_{J\subseteq K}(-1)^{|K|-|J|}(|J|+2)\alpha(\tau_{n}(J))\hat{!}.$$
\end{corolario}

\begin{proof}
From the corollary above, we know \begin{gather*}\sum_{J\subseteq[2,n+1]\smallsetminus K}(-1)^{n-|K|-|J|}(\alpha(J^{\max})\hat{!}-\alpha(J)\hat{!})=\\\sum_{J\subseteq K}(-1)^{|K|-|J|}(|J|+1)\alpha(\tau_{n}(J))\hat{!}\end{gather*} but, by the initial corollary, we also know \begin{gather*}
    \sum_{J\subseteq[2,n+1]\smallsetminus K}(-1)^{n-|K|-|J|}(\alpha(J)\hat{!})=\sum_{J\subseteq K}(-1)^{|K|-|J|}\alpha(\tau_{n}(J))\hat{!},
\end{gather*} which together establish the identity we wanted to prove.
\end{proof}

Eulerian polynomials can also be written through excedances. We see how this allows us to get more interesting combinatorial results in the following section.

\section{Mirrorpalindromicity and excedances}

First we establish precisely the expression of Eulerian polynomials through excedances.

\begin{remark}[Riordan's bijection]\cite[Equation 4.3]{branden2011proof}\label{Riordanremark}
Using a bijection of Riordan, it was noted that Eulerian polynomials can be written in terms of excedances $A_{n}(\mathbf{x}):=$ $$\sum_{\sigma\in\mathfrak{S}_{n+1}}\prod_{i\in\exc(\sigma)}x_{i},$$ where $\exc(\sigma):=\{\sigma(i)\in[n+1]\mid\sigma(i)>i\}\subseteq[2,n+1]$ is the set of \textit{excedances} of $\sigma$. In imitation of the previous notation, we denote $r(n,X):=\{\sigma\in\mathfrak{S}_{n+1}\mid\exc(\sigma)=X\}.$
\end{remark}

As a consequence of this rewriting of multivariate Eulerian polynomials, we immediately obtain duplicates of the corollaries in the section above. This is the content of the next remark.

\begin{remark}[Twin corollaries]
An immediate consequence of this expression is the fact that $|r(n,X)|=|R(n,X)|$ and therefore we obtain as immediate corollaries results similar to the ones above but for excedances. The corresponding bijection can be constructed composing the bijection given by Riordan in \cite{riordan2014introduction} with the one we used in the proof of Theorem \ref{especialita}.
\end{remark}

In order to analyze how special is the property proved in this theorem that motivated our last discussion and corollaries, we introduce a couple of concepts about (multiaffine) polynomials in the next section.

\section{Conditions, obstructions and conclusion}

We finally study a limitation of polynomials to be mirrorpalindromic in order to understand better how special are our Eulerian polynomials in this regard. In particular, we look at complete polynomials in the sense that all its monomials play a role.

\begin{definicion}[(Degree-)completeness]
A multiaffine polynomial strictly in $n$ variables is \textit{complete} if it has nonzero coefficients multiplying all its possible monomials. A relaxation of this concept is considering \textit{degree-complete} polynomials in which all possible degrees from $1$ to $n$ appear at least once (with a nonzero coefficient).
\end{definicion}

This form of completeness is present in Eulerian polynomials.

\begin{proposicion}[Completeness of Eulerian polynomials]
Eulerian polynomials are complete.
\end{proposicion}

\begin{proof}
Let $m$ be one arbitrary possible monomial and order its variables in increasing order of its indices so we can write $x_{i_{1}}<\dots<x_{i_{d}}$ with $d\leq n$ and $1<i_{1}<\dots<i_{d}$ remembering that $x_{1}$ is never a variable because that index can never be a descent top. For the indices not in the previous monomial (and therefore corresponding to the variables of the polynomial not appearing in $m$ and different from $x_{1}$), order them also in increasing order $1<j_{1}<\dots<j_{s}$ with $s+d=n$. Now, consider the permutation (written in one-line notation) $\sigma:=(i_{d}\cdots i_{1}\cdots1 j_{1}\cdots j_{s})$. This sigma ensures the appearance of the given monomial $m$ and, as this monomial $m$ was arbitrary among all the possible monomials and there are no cancellations on the definition of the Eulerian polynomials, this finishes our proof.
\end{proof}

One can ask then if Theorem \ref{especialita} is true for all multiaffine monomialmaximal complete polynomials. The answer is clearly no as the next example shows.

\begin{ejemplo}[Multiaffine monomialmaximal does not imply mirrorpalindromic]
Consider the multiaffine monomialmaximal complete polynomial $$A(\mathbf{x}):=1+2x_{1}+x_{2}+x_{3}+3x_{2}x_{3}+x_{1}x_{2}+x_{1}x_{3}+x_{1}x_{2}x_{3}.$$ Its reciprocal is $$\rec(A(\mathbf{x}))=x_{1}x_{2}x_{3}+2x_{2}x_{3}+x_{1}x_{3}+x_{1}x_{2}+3x_{1}+x_{3}+x_{2}+1$$ and it is clear that no permutation of the variables can transform the degree $1$ part $3x_{1}+x_{3}+x_{2}$ of $\rec(A)$ into the degree $1$ part $2x_{1}+x_{2}+x_{3}$ of $A$. Thus $A$ is not mirrorpalindromic.
\end{ejemplo}

This finishes our discussion about the polynomial identity described in Theorem \ref{especialita} and its immediate combinatorial consequences in terms of relations between cardinals of some sets of permutations and sums involving subsets of a certain set. This also finishes our study of symmetries within multivariate Eulerian polynomials

Going further lies beyond our scope in this paper but it is clear that more symmetries of this kind can be unveiled in the lifting from univariate polynomials to multivariate ones. The importance of these symmetries in order to understand more about these (and related) families of polynomials and their  underlying combinatorial objects is a clear indication to continue researching in this direction.

\printbibliography
\end{document}